\documentclass[letterpaper, 10pt, conference]{ieeeconf}

\IEEEoverridecommandlockouts

\usepackage[utf8]{inputenc}
\usepackage[english]{babel}

\usepackage{graphicx}
\graphicspath{{Figures/}}
\usepackage{subcaption}

\usepackage{newtxmath, newtxtext}

\usepackage{mathtools, amssymb, amsthm, bm}
\newtheorem{theorem}{Theorem}
\newtheorem{lemma}{Lemma}

\newtheorem{corollary}{Corollary}
\theoremstyle{definition}

\newtheorem{problem}{Problem}
\newtheorem{assumption}{Assumption}
\theoremstyle{remark}

\let\labelindent\relax

\usepackage{array}
\usepackage{enumerate}
\usepackage{enumitem}
\usepackage{cite}
\usepackage{xcolor}
\usepackage{multirow}
\usepackage{booktabs}
\usepackage{balance}

\usepackage{hyperref}
\hypersetup{colorlinks = true, linkcolor = blue, citecolor = blue, urlcolor = blue}

\allowdisplaybreaks

\DeclareMathOperator*{\U}{\mathcal{U}}
\DeclareMathOperator*{\Uc}{\mathcal{U}^c}
\DeclareMathOperator*{\Uuc}{\mathcal{U}^{uc}}
\DeclareMathOperator*{\X}{\mathcal{X}}

\title{\LARGE \bf Finite-time Reachability for\\Constrained, Partially Uncontrolled Nonlinear Systems
}

\author{Ram Padmanabhan and Melkior Ornik
\thanks{The authors are with the University of Illinois Urbana-Champaign, Urbana, IL 61801, USA. Emails: \{ramp3, mornik\}@illinois.edu}%
\thanks{This work was supported in part by the Air Force Office of Scientific Research under Grant FA9550-23-1-0131 and in part by the NASA University Leadership Initiative under Grant 80NSSC22M0070. Corresponding Author: Ram Padmanabhan.}%
}

\begin{document}

\maketitle

\begin{abstract}
This paper presents a technique to drive the state of a constrained nonlinear system to a specified target state in finite time, when the system suffers a partial loss in control authority. Our technique builds on a recent method to control constrained nonlinear systems by building a simple, linear driftless approximation at the initial state. We construct a partition of the finite time horizon into successively smaller intervals, and design controlled inputs based on the approximate dynamics in each partition. Under conditions that bound the length of the time horizon, we prove that these inputs result in bounded error from the target state in the original nonlinear system. As successive partitions of the time horizon become shorter, the error reduces to zero despite the effect of uncontrolled inputs. A simulation example on the model of a fighter jet demonstrates that the designed sequence of controlled inputs achieves the target state despite the system suffering a loss of control authority over one of its inputs.
\end{abstract}

\section{Introduction} \label{sec:Introduction}
The challenges in controlling nonlinear dynamical systems are well-documented. While there exist a number of techniques to control such systems, each of these suffers from certain disadvantages that may limit their applicability. For instance, feedback linearization methods require an invertible coordinate transformation that may not exist \cite{Khalil_NS}. Sliding-mode control may suffer from poor performance especially in the presence of actuation constraints \cite{SMC1, SMC2}, and methods such as nonlinear model predictive control require extensive computations \cite{NMPC}. These challenges are compounded when such systems are subjected to undesirable effects such as exogenous disturbances, modeled and unmodeled uncertainties. Such undesirable effects are often a consequence of operating in uncertain or adversarial environments, and may cause a control system to fail its objectives.

The focus of this paper is on nonlinear control systems that suffer a loss in control authority over a subset of their actuators, thus becoming partially uncontrolled. This effect is practically motivated by the example of the Nauka research laboratory module, which suffered a loss in attitude control while docking to the International Space Station in 2021 \cite{Nauka}. A partial loss in control authority separates inputs into \emph{controlled} and \emph{uncontrolled} components. Controlled inputs must be designed to compensate for the effect of uncontrolled inputs, which may take on any values in their admissible set and may be chosen by an adversary.

A partial loss in control authority can also be modeled under the paradigm of robust control theory \cite{RC1, RC2}, which focuses on systems that are subjected to external perturbations. These external perturbations may model uncontrolled inputs that can be chosen by an adversary to prevent a control system from achieving its tasks. In the setting of nonlinear systems, the design of robust control laws is a well-studied problem. While the literature is far too extensive to adequately review here, a number of classical approaches \cite{Qu98, WS02, LBS92, WXD92} involve Lyapunov-based analysis that guarantee minimization of robustness metrics, such as $\mathcal{H}_2$ or $\mathcal{H}_{\infty}$ norms \cite{RC1}. At the same time, such approaches to nonlinear robust control often assume unconstrained inputs and external perturbations that affect the system in a linear manner. In contrast, the setting of a loss of control authority imposes actuation constraints on controlled and uncontrolled inputs. Further, and more importantly, uncontrolled inputs enter the system dynamics after being acted on by a nonlinear function of the state. A direct application of classical methods in robust control would require considering constraints on uncontrolled inputs that are state-dependent and potentially non-convex. Such problems are well-known to be difficult to address \cite{LCF22, CVOC}.

Recent work in the setting of a partial loss of control authority has introduced quantitative metrics to analyze the effect of such a malfunction. These metrics have aimed to quantify the maximal additional time \cite{BXO21, BO23} or energy \cite{PO25a, PO26a} used by a system to achieve a target state under this malfunction. While there has been some effort towards designing controlled inputs to achieve a task despite any adverse effects from uncontrolled inputs \cite{BO22b}, that work considers only linear dynamical systems which are simpler to analyze. This paper addresses this gap by considering nonlinear systems.

We present a method to control nonlinear systems that suffer a partial loss in control authority, in the presence of actuation constraints and over a finite horizon. Our method is based on a constrained nonlinear controller proposed in \cite{PO25b}, where a linear driftless approximation to the original nonlinear system is constructed. It is proved in \cite{PO25b} that a sequence of optimal control inputs designed for this approximation asymptotically stabilizes the nonlinear system to a target state, under an appropriate partition of the time horizon. We extend the method in \cite{PO25b} to the finite-horizon setting and for nonlinear systems that suffer a partial loss in control authority. We construct a similar linear driftless approximation, and under a partition of the finite horizon, propose a sequence of controlled inputs in each interval of this partition. 
Under conditions that bound the length of the time horizon, we prove that the controlled inputs satisfy actuation constraints, and the length of each interval in the partition directly impacts the error from the target state. As intervals in the partition become shorter, the error is driven to zero. We present a simulation example on the model of a fighter jet losing authority over one of its inputs, illustrating the proposed framework.

The remainder of this paper is organized as follows. Section~\ref{sec:Formulation} introduces the problem we intend to solve and recalls the notion of the linear driftless approximation. Section~\ref{sec:Method} presents the sequence of controlled inputs that we design, and section~\ref{sec:Convergence} proves that this method solves the proposed problem. Section~\ref{sec:Convergence} also derives conditions under which actuation constraints are satisfied by the controlled input. Section~\ref{sec:Example} presents an example of a model of a fighter jet that loses authority over one of its inputs, and shows how the proposed sequence of controlled inputs achieves the target state despite uncontrolled effects.

\section{Problem Formulation} \label{sec:Formulation}
Consider a nonlinear dynamical system that evolves on a state space $\X$, where $\X$ is a compact subset of $\mathbb{R}^d$:
\begin{equation} \label{eq:Nominal}
	\dot{x}(t) = f(x(t)) + g(x(t))u(t), ~x(0) = x_0 \in \X
\end{equation}
where the control input $u$ lies in the admissible set $\U$, defined as
\begin{align}
\U \coloneqq \{&u: \mathbb{R}^+ \to \mathbb{R}^{m+p}: u \text{ is piecewise continuous in $t$, } \nonumber \\
&\|u(t)\|_{\infty} \leq 1 \text{ for all $t$}\}, \label{eq:SetU}
\end{align}
and the functions $f:\mathbb{R}^d \to \mathbb{R}^d$ and $g:\mathbb{R}^d \to \mathbb{R}^{m+p}$ are $D_f$- and $D_g$-Lipschitz continuous in the $\infty$-norm on the state space $\X$. Then, there exist constants $D_f$ and $D_g$ such that for all $x_1, x_2 \in \X$,
\begin{align}
f(x_1) - f(x_2) &= d_f(x_1, x_2), \|d_f(x_1, x_2)\|_{\infty} \leq D_f \|x_1 - x_2\|_{\infty}; \label{eq:f_cond} \\
g(x_1) - g(x_2) &= d_g(x_1, x_2), \|d_g(x_1, x_2)\|_{\infty} \leq D_g \|x_1 - x_2\|_{\infty}. \label{eq:g_cond}
\end{align}
{Note that Lipschitz continuity on $\X$ is equivalent to local Lipschitz continuity if $\X$ is compact, by the finite subcover property of compact sets \cite{Lipschitz}.}

We assume the system \eqref{eq:Nominal} undergoes a malfunction causing it to lose authority over $p$ of its $m+p$ actuators:
\begin{equation} \label{eq:Malfunctioning}
\dot{x}(t)\!=\!f(x(t)) + g^c(x(t))u^c(t) + g^{uc}(x(t))u^{uc}(t), ~x(0) = x_0\!\in\!\X
\end{equation}
where the input $u$ splits into \emph{controlled} and \emph{uncontrolled} components $u^c$ and $u^{uc}$ respectively. Then, $u^c$ and $u^{uc}$ lie in admissible sets $\Uc$ and $\Uuc$ which are defined as
\begin{align}
\Uc \coloneqq \{&u^c: \mathbb{R}^+ \to \mathbb{R}^{m}: u^c \text{ is piecewise continuous in $t$, } \nonumber \\
&\|u^c(t)\|_{\infty} \leq 1 ~\text{for all $t$}\}, \label{eq:SetUc} \\
\Uuc \coloneqq \{&u^{uc}: \mathbb{R}^+ \to \mathbb{R}^{p}: u^{uc} \text{ is piecewise continuous in $t$, } \nonumber \\
&\|u^{uc}(t)\|_{\infty} \leq 1 ~\text{for all $t$}\}. \label{eq:SetUuc}
\end{align}
The uncontrolled input $u^{uc} \in \Uuc$ represents either actuator faults or those suffering an adversarial attack. We are interested in the problem of achieving a target state in fixed time despite such a loss in control authority. This problem is formally stated below.

\begin{problem} \label{prob}
Given an initial state $x_0 \in \mathbb{R}^d$, a target state $x_{tg} \in \mathbb{R}^d$ a fixed final time $t_f$ and a prescribed radius $\varepsilon > 0$, design the controlled input $u^c \in \Uc$ such that $\|x_{tg} - x(t_f)\|_{\infty} \leq \varepsilon$ {for any $u^{uc} \in \Uuc$.}
\end{problem}


Define $g_0 \coloneqq g(x_0)$, $g^{c}_{0} \coloneqq g^c(x_0)$ and $g^{uc}_{0} \coloneqq g^{uc}(x_0)$. Then, we say that the dynamics
\begin{equation} \label{eq:DftApprox}
\dot{x}(t) = g^{c}_{0}u^c(t) + g^{uc}_{0}u^{uc}(t), ~x(0) = x_0
\end{equation}
are the \emph{linear driftless approximation} to \eqref{eq:Malfunctioning}. We make the following assumption on $g^{c}_{0}$, used in designing control laws in Section~\ref{sec:Method}.

\begin{assumption}
The matrix $g^{c}_{0}$ has full row rank $d$.
\end{assumption}

This assumption may be relaxed, but may result in some bounded steady-state error as discussed in \cite{PO25b}. Adapting our proposed method for any $g^{c}_{0}$ is an important avenue for future work. In the following section, we present our method to solve Problem~\ref{prob} for the system \eqref{eq:Malfunctioning}.

\begin{figure}[!t]
	\centering
	\includegraphics[width = 0.5\textwidth]{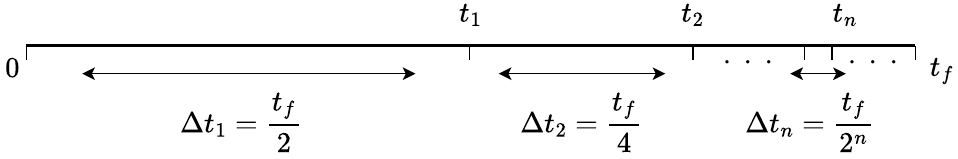}
	\caption{Partitioning the interval $[0,t_f]$ using the sequence $\{t_n\}$.}
	\label{fig:Horizon}
\end{figure}

\section{Method} \label{sec:Method}
The key idea in our method is to partition the time horizon $[0, t_f]$ into successively smaller partitions as follows. Construct the sequence of time instants $\{t_n\}$ where $t_0 = 0$ and $t_n = \left(\frac{2^n-1}{2^n}\right)t_f$, so that $t_1 = \frac{t_f}{2}$, $t_2 = \frac{3t_f}{4}$, and so on, as shown in Fig.~\ref{fig:Horizon}. Let $\Delta t_n = t_n - t_{n-1} = \frac{t_f}{2^n}$.  The sequence $\{t_n\}$ is clearly a geometric sequence with $\lim_{n \to \infty} t_n = t_f$ and $\lim_{n \to \infty} \Delta t_n = 0$.

Let $x_n \coloneqq x(t_n) \in \X$ and $\tilde{x}_n = x_n - x_{tg}$, denoting the error at $t_n$. The control law we propose is written as follows:
\begin{equation} \label{eq:ucn}
u^{c}_{n}(t) = -\frac{1}{\Delta t_n} {g^{c}_{0}}^{\dagger}\left(\tilde{x}_{n-1} + g^{uc}_{0}\alpha_n\right), ~t \in [t_{n-1}, t_n]
\end{equation}
for $n = 1, 2, \ldots$ and where $\alpha_n \in \mathbb{R}^p$ is a vector-valued sequence to be designed. The controlled input $u^c(t)$ is then formed by applying $u^{c}_{n}(t)$ in its corresponding interval $[t_{n-1}, t_n]$. Clearly, $u^c(t)$ is piecewise constant, and is motivated by a similarly structured optimal control law for linear driftless systems discussed in \cite{PO25a}. We remark that in \cite{PO25a}, $u^{c}(t)$ depended on the uncontrolled input $u^{uc}(t)$, whereas we replace this dependence by the sequence $\alpha_n$ in \eqref{eq:ucn}. The undesired effects from $u^{uc}(t)$ are compensated using the state through $\tilde{x}_{n-1} = x_{n-1}-x_{tg}$.

Applying \eqref{eq:ucn} to \eqref{eq:Malfunctioning} over $[t_{n-1}, t_n]$, 
\begin{align}
x_n - x_{n-1} &= \int_{t_{n-1}}^{t_n} f(x(\tau))\mathrm{d}\tau + \int_{t_{n-1}}^{t_n} d_g(x(\tau),x_0)u_n(\tau)\mathrm{d}\tau \nonumber \\
&+ g^{c}_{0} \int_{t_{n-1}}^{t_n} u^{c}_{n}(\tau)\mathrm{d}\tau + g^{uc}_{0} \int_{t_{n-1}}^{t_n} u^{uc}(\tau)\mathrm{d}\tau \nonumber \\
&= -v(t_{n-1}, t_n) - \tilde{x}_{n-1} - g^{uc}_{0}\alpha_n + g^{uc}_{0} \Delta t_n \overline{u}^{uc}, \label{eq:Sol11}
\end{align}
where we define 
\begin{equation} \label{eq:v_def}
	v(t_a, t_b) \coloneqq -\int_{t_a}^{t_b} f(x(\tau)) \mathrm{d}\tau - \int_{t_a}^{t_b} d_g(x(\tau),x_0)u(\tau)\mathrm{d}\tau
\end{equation}
for the input $u(t)$ in the interval $[t_a,t_b]$. We also define $\overline{u}^{uc} \coloneqq \frac{1}{\Delta t_n} \int_{t_{n-1}}^{t_n} u^{uc}(\tau)\mathrm{d}\tau$ as the mean value of $u^{uc}(t)$ in $[t_{n-1},t_n]$. In \eqref{eq:Sol11}, we use the fact that $u^{c}_{n}(t)$ is constant in the interval $[t_{n-1},t_n]$ in \eqref{eq:ucn}, and thus $g^{c}_{0} \int_{t_{n-1}}^{t_n} u^{c}_{n}(\tau)\mathrm{d}\tau = -\tilde{x}_{n-1} - g^{uc}_{0}\alpha_n$. In the expression for $v(t_{n-1},t_n)$, 
\begin{equation} \label{eq:un_gen}
	u_n(\tau) = \begin{bmatrix} u^{c}_{n}(\tau) \\ u^{uc}(\tau) \end{bmatrix},
\end{equation}
consisting of both controlled and uncontrolled components. Substituting $\tilde{x}_{n-1}$ and rearranging,
\begin{equation} \label{eq:Sol1}
x_{tg} - x_n = v(t_{n-1}, t_n) + g^{uc}_{0}\left(\alpha_n - \Delta t_n \overline{u}^{uc}\right).
\end{equation}

In the following section, we show that the right-hand side of \eqref{eq:Sol1} converges to zero as $n$ increases. In this proof, we use the fact that $\Delta t_n$ converges to zero, and design $\alpha_n$ to converge to zero at the same rate as $\Delta t_n$, i.e.,
\begin{equation} \label{eq:alpha_n}
	\alpha_n = \frac{1}{2^n}\mathbf{1}_p,
\end{equation}
where $\mathbf{1}_p$ denotes the $p$-dimensional vector of ones. This choice of $\alpha_n$ is not unique, but we restrict our analysis to this case. Investigating more general choices of $\alpha_n$ is an important avenue for future work.

We remark that in practice, the sequence of inputs \eqref{eq:ucn} would only be applied only until some $n = \overline{n}$, where $\overline{n}$ depends on $\varepsilon$ in Problem~\ref{prob}. The final input $u^{c}_{\overline{n}}$ would then be applied for the entire remaining interval $[t_{\overline{n}-1}, t_f]$ rather than the shorter interval $[t_{\overline{n}-1}, t_{\overline{n}}]$. We also discuss the choice of such $\overline{n}$ in the following section. This choice avoids issues that occur when control inputs \eqref{eq:ucn} change with increasing frequency as $\Delta t_n\!\to\!0$. Fast changes in control inputs can cause harm to physical components such as actuators in practical systems.


\section{Proof of Convergence} \label{sec:Convergence}
To prove convergence of the strategy proposed in Section~\ref{sec:Method}, we first let $D_S = D_f + D_g$, and define the following constants:
\begin{align}
c &\coloneqq \|f(x_0)\|_{\infty} + \|g^{uc}_{0}\|_{\infty} + D_S \|x_{tg} - x_0\|_{\infty}, \label{eq:c} \\
c_1 &\coloneqq 4c\|{g^{c}_{0}}^{\dagger}\|_{\infty}, \label{eq:c1}, \\
\text{and } c_2 &\coloneqq 4D_S \|{g^{c}_{0}}^{\dagger}\|_{\infty} \|g^{uc}_{0}\|_{\infty} \|\alpha_2\|_{\infty} \label{eq:c2}
\end{align}
where $\alpha_2$ is obtained from \eqref{eq:alpha_n}. {The proof of convergence presented in this section relies on a number of conditions being satisfied by these constants, which is not guaranteed. However, the analysis is also quite conservative and as we remark later, violating such conditions does not inhibit satisfactory performance.} We next present the following lemmas which are used in the proof of convergence.

\begin{lemma} \label{lem:err}
Consider the sequence
\begin{equation} \label{eq:errn}
\overline{e}_n \coloneqq \frac{c}{D_S} \left(e^{\Delta t_n D_S}-1\right).
\end{equation}
Then, $\overline{e}_{n+1} \leq \frac{1}{2} \overline{e}_n$ and $\lim_{n\to\infty} \overline{e}_n = 0$. Given some $\varepsilon > 0$, there thus exists some $n_1$ such that $\overline{e}_{n_1} \leq \varepsilon$.
\end{lemma}

\begin{proof}
Recall the properties $\Delta t_{n+1} = \frac{1}{2}\Delta t_n$, $\lim_{n\to\infty} \Delta t_n = 0$. Using a Taylor series expansion, we know
\begin{align*}
e^{\Delta t_{n+1}D_S} - 1 &= \sum_{k = 1}^{\infty} \frac{(\Delta t_{n+1}D_S)^k}{k!} = \sum_{k = 1}^{\infty} \left(\frac{1}{2}\right)^k \frac{(\Delta t_nD_S)^k}{k!} \\
&\leq \frac{1}{2} \sum_{k = 1}^{\infty} \frac{(\Delta t_nD_S)^k}{k!} = \frac{1}{2} \left(e^{\Delta t_nD_S}-1\right),
\end{align*}
since $\left(\frac{1}{2}\right)^k \leq \frac{1}{2}$ for all $k \geq 1$ and the summation consists of only positive terms. From the definition of $\overline{e}_n$ in \eqref{eq:errn}, the result $\overline{e}_{n+1} \leq \frac{1}{2} \overline{e}_n$ follows, and clearly $\lim_{n \to \infty} \overline{e}_n = 0$. Thus, given an $\varepsilon > 0$, there exists some $n_1$ such that $\overline{e}_{n_1} \leq \varepsilon$.
\end{proof}

\begin{figure}[!t]
	\centering \hspace{-0.5cm}
	\includegraphics[width = 0.42\textwidth]{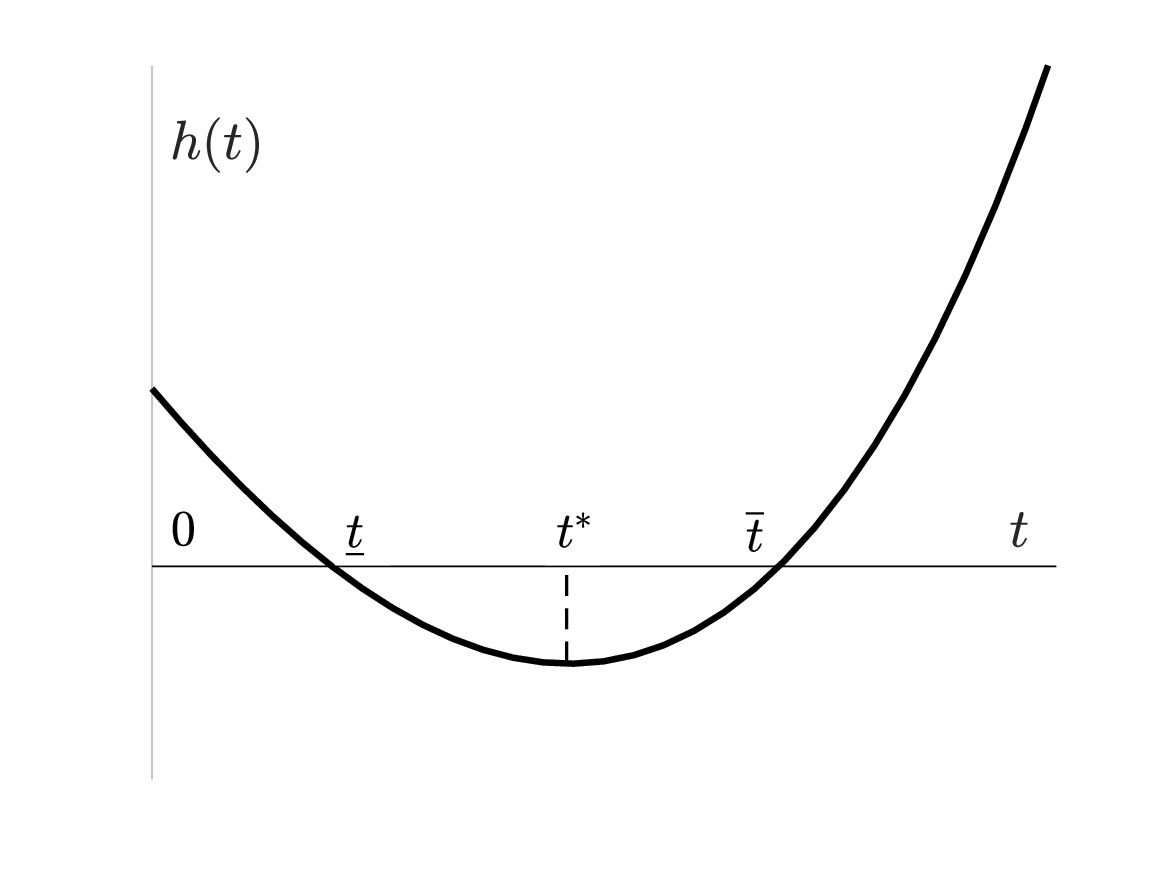}
	\caption{The function $h(t)$ and its behavior based on $\frac{dh(t)}{dt}$.}
	\label{fig:h}
	\vspace{-0.3cm}
\end{figure}

\begin{lemma} \label{lem:h}
Consider the function
\begin{equation} \label{eq:h}
	h(t) = e^{tD_S/2} - 1 - \frac{tD_S - c_2}{c_1}
\end{equation}
for $t \geq 0$. {Then, there exists an interval $[\underline{t}, \overline{t}]$ where $h(t) \leq 0$ if and only if the conditions
\begin{equation} \label{eq:ccond}
	c_1 < 2, ~~~ c_2 < c_1 + 2\left(\log(2/c_1) - 1\right)
\end{equation}
are satisfied, where $\underline{t}$ and $\overline{t}$ satisfy $h(\underline{t}) = h(\overline{t}) = 0$.}
\end{lemma}

\begin{proof}
First, note that $h(0) = \frac{c_2}{c_1} \geq 0$ from \eqref{eq:c1}, \eqref{eq:c2}. For $h(t)$ to take on negative values, we require its derivative to be negative so that $h(t)$ decreases sufficiently. Differentiating $h(t)$,
\begin{equation} \label{eq:dhdt}
\frac{dh(t)}{dt} = D_S \left(\frac{1}{2} e^{tD_S/2} - \frac{1}{c_1}\right).
\end{equation}
We thus require $c_1 < 2$ so that $\frac{dh(t)}{dt} < 0$ for at least some values of $t$, before increasing to arbitrarily large positive values. We also note that there exists only one point $t^*$ where $\frac{dh(t)}{dt}\big\vert_{t^*} = 0$. If we have $h(t^*) < 0$, then there exists some range $t \in [\underline{t}, \overline{t}]$ where $h(t) \leq 0$. The critical point $t^*$ is obtained by
\begin{equation} \label{eq:tstar}
\frac{dh(t)}{dt}\Big\vert_{t^*} = 0 \implies e^{t^*D_S/2} = \frac{2}{c_1} \text{ or } t^* = \frac{2}{D_S}\log\left(\frac{2}{c_1}\right).
\end{equation}
Then,
$$
h(t^*) = \frac{2}{c_1} - 1 - \frac{2\log(2/c_1) - c_2}{c_1}
$$
and on rearranging, we see that $h(t^*) < 0$ if and only if
$$
	c_1 < 2, ~~~ c_2 < c_1 - 2\left(1 - \log(2/c_1)\right),
$$
thus proving \eqref{eq:ccond}. Under conditions \eqref{eq:ccond}, the function $h(t)$ takes on non-positive values in some interval $[\underline{t}, \overline{t}]$. This behavior of $h(t)$ is illustrated in Fig.~\ref{fig:h}.
\end{proof}

We are now ready to state the central result of this paper.

\begin{theorem}[Boundedness] \label{thm:bound}
Assume \eqref{eq:ccond} is satisfied. Then, the sequence of controlled inputs $u_{n}^{c}$ in \eqref{eq:ucn} satisfies the constraint $u_{n}^{c} \in \Uc$ in \eqref{eq:SetUc} if $t_f$ satisfies the conditions
\begin{subequations} \label{eq:tf_cond}
\begin{align}
t_f &\geq 2\|{g^{c}_{0}}^{\dagger}\|_{\infty} \Big[\|\tilde{x}_0\|_{\infty} + \frac{1}{2}\left\|g^{uc}_{0}\right\|_{\infty}\Big], \\
\text{and } t_f &\in \left[\underline{t}, \overline{t}\right],
\end{align}
\end{subequations}
where $\left[\underline{t}, \overline{t}\right]$ is the interval where $h(t) \leq 0$ in Lemma~\ref{lem:h}. Under these conditions, for all $n \geq 1$, $x_{tg} - x_n$ is bounded by
\begin{equation} \label{eq:en}
\|x_{tg} - x_n\|_{\infty} \leq \overline{e}_n \coloneqq \frac{c}{D_S} \left(e^{\Delta t_n D_S} - 1\right).
\end{equation}
\end{theorem}

\begin{proof}

The proof is organized into two parts. In the first part, we consider the case of $n = 1$, proving one part of condition \eqref{eq:tf_cond} and the bound \eqref{eq:en}. In this part, we recall key arguments from \cite[Theorem 1]{PO25b}. In the second part, we consider $n > 1$, proving the rest of condition \eqref{eq:tf_cond} and the general bound \eqref{eq:en}.

\noindent \emph{\underline{Part 1: Proof for $n = 1$:}} Consider
\begin{equation} \label{eq:uc1}
u^{c}_{1}(t) = -\frac{1}{\Delta t_1} {g^{c}_{0}}^{\dagger} \left(\tilde{x}_0 + g_{0}^{uc}\alpha_1\right),
\end{equation}
where we know that $\Delta t_1 = t_1 = t_f/2$. Rearranging, we note that if the condition
\begin{equation} \label{eq:tf1}
	t_f \geq 2\|{g^{c}_{0}}^{\dagger}\|_{\infty} \left[\|\tilde{x}_0\|_{\infty} + \frac{1}{2}\left\|g^{uc}_{0}\right\|_{\infty}\right]
\end{equation}
is satisfied, then $u^{c}_{1} \in \Uc$, i.e., $\|u_{1}^{c}(t)\|_{\infty} \leq 1$ for all $t$, using $\|\alpha_1\|_{\infty} = \frac{1}{2}$ from \eqref{eq:alpha_n}. This proves one part of condition \eqref{eq:tf_cond}.

Next, we consider 
\begin{equation} \label{eq:v0t}
v(0, t) = - \int_{0}^{t} f(x(\tau))\mathrm{d}\tau - \int_{0}^{t} d_g(x(\tau),x_0)u_1(\tau)\mathrm{d}\tau
\end{equation}
for some $t \in [0,t_1]$, according to \eqref{eq:v_def}. We first write $f(x(\tau)) = f(x_0) + d_f(x(\tau),x_0)$, and note that condition \eqref{eq:tf1} and the restriction $u^{uc} \in \Uuc$ on the uncontrolled input imply $\|u_1(\tau)\|_{\infty} \leq 1$ for all $\tau$, where $u_1(\tau)$ is defined through \eqref{eq:un_gen}. We then take norms on both sides of \eqref{eq:v0t}, use Jensen's inequality to move the norm inside integrals, use properties \eqref{eq:f_cond}, \eqref{eq:g_cond} and condition \eqref{eq:tf1} and recall $D_S = D_f + D_g$ to obtain
\begin{align}
\|v(0,t)\|_{\infty} &\leq t\|f(x_0)\|_{\infty} + D_S \int_{0}^{t}\|x(\tau) - x_0\|_{\infty} \mathrm{d}\tau \nonumber \\
&\leq t\left(\|f(x_0)\|_{\infty} + D_S\|x_{tg} - x_0\|_{\infty} \right) \nonumber \\
&+ D_S \int_{0}^{t} \|x_{tg} - x(\tau)\|_{\infty} \mathrm{d}\tau, \label{eq:v01}
\end{align}
where we use the triangle inequality of norms in the second line. We now attempt to bound the last term in \eqref{eq:v01}. Writing out the solution to \eqref{eq:Malfunctioning} for $t \in [0, t_1]$, we know
\begin{align*}
	x(t) - x_0 &= \int_{0}^{t} f(x(\tau))\mathrm{d}\tau + \int_{0}^{t} d_g(x(\tau),x_0) u_1(\tau) \mathrm{d}\tau \nonumber \\
	& + g^{c}_{0} \int_{0}^{t} u^{c}_{1}(\tau)\mathrm{d}\tau + g^{uc}_{0} \int_{0}^{t} u^{uc}(\tau)\mathrm{d}\tau \nonumber \\
	&= -v(0, t) - \frac{t}{\Delta t_1} \tilde{x}_0 + g^{uc}_{0} \left[\int_{0}^{t} u^{uc}(\tau) \mathrm{d}\tau - \frac{t}{\Delta t_1}\alpha_1\right],
\end{align*}
where we have substituted $u^{c}_{1}(\tau)$ from \eqref{eq:uc1}, noting that $g^{c}_{0}u_{1}^{c}(t) = -\frac{1}{\Delta t_1}\left(\tilde{x}_0 + g^{uc}_{0}\alpha_1\right)$. Subtracting $x_{tg}$ on both sides and rearranging, 
\begin{equation} \label{eq:err_t1}
	x_{tg} - x(t) = v(0,t) - \frac{t_1-t}{\Delta t_1} \tilde{x}_0 + g^{uc}_{0}\left[\frac{t_1-t}{\Delta t_1}\alpha_1 - \int_{0}^{t} u^{uc}(\tau)\mathrm{d}\tau \right].
\end{equation}
Taking norms on both sides and noting that $\frac{t_1-t}{\Delta t_1} \leq 1$ for $t \in [0,t_1]$, we have
\begin{align}
\|x_{tg} - x(t)\|_{\infty} &\leq \|v(0,t)\|_{\infty} + \|\tilde{x}_0\|_{\infty} \nonumber \\
&+ \left\|g^{uc}_{0} \left[\frac{t_1-t}{\Delta t_1}\alpha_1 - \int_{0}^{t} u^{uc}(\tau)\mathrm{d}\tau \right]\right\|_{\infty} \nonumber \\
&\leq \|v(0,t)\|_{\infty} + \|\tilde{x}_0\|_{\infty} \nonumber \\
&+ \|g^{uc}_{0}\|_{\infty} \left[\|\alpha_1\|_{\infty} + \int_{0}^{t} \|u^{uc}(\tau)\|_{\infty} \mathrm{d}\tau\right] \nonumber \\
&\leq \|v(0,t)\|_{\infty} + \|\tilde{x}_0\|_{\infty} + \|g^{uc}_{0}\|_{\infty} \left(t + \frac{1}{2}\right), \label{eq:err_t2}
\end{align}
where we use Jensen's inequality \cite{NPBook} to move norms inside integrals once more, the constraint \eqref{eq:SetUuc} and the fact that $\|\alpha_n\|_{\infty} \leq \frac{1}{2}$ for all $n$ from \eqref{eq:alpha_n}. We now substitute \eqref{eq:v01} in \eqref{eq:err_t2} and obtain
\begin{align}
\|x_{tg} - x(t)\|_{\infty} &\leq t\underbrace{\left(\|f(x_0)\|_{\infty} + \|g^{uc}_{0}\|_{\infty} + D_S \|x_{tg} - x_0\|_{\infty}\right)}_{= c} \nonumber \\
+ &\|\tilde{x}_0\|_{\infty} + \frac{1}{2}\|g^{uc}_{0}\|_{\infty} + D_S \int_{0}^{t} \|x_{tg} - x(\tau)\|_{\infty} \mathrm{d}\tau, \label{eq:err_t3}
\end{align}
recalling $c$ from \eqref{eq:c}.

Let $Q^1(t) \coloneqq \int_{0}^{t} \|x_{tg} - x(\tau) \|_{\infty} \mathrm{d}\tau$, so that $\dot{Q}^1(t) = \|x_{tg} - x(t)\|_{\infty}$ \cite{Apostol}. Then, \eqref{eq:err_t3} can be rewritten as
\begin{equation} \label{eq:Q1}
\dot{Q}^1(t) \leq ct + \|\tilde{x}_0\|_{\infty} + \frac{1}{2}\|g^{uc}_{0}\|_{\infty} + D_SQ^1(t).
\end{equation}
We now invoke similar arguments to those made in \cite[Theorem 1]{PO25b} which involve Gr{\"o}nwall's inequality \cite{Khalil_NS}. These arguments can be used to prove that
$$
\|x_{tg} - x(t)\|_{\infty} \leq \frac{c}{D_S} \left(e^{tD_S} - 1 \right), ~ t \in [0, t_1],
$$
and in particular,
\begin{equation} \label{eq:e1}
\|x_{tg} - x_1\|_{\infty} \leq \overline{e}_1 \coloneqq \frac{c}{D_S} \left(e^{\Delta t_1 D_S} - 1\right).
\end{equation}
Equation \eqref{eq:e1} proves the bound \eqref{eq:en} for $n = 1$, bounding the error from $x_{tg}$ at time $t_1$.

\noindent \emph{\underline{Part 2: Proof for $n > 1$:}} We sketch the remainder of the proof for $n > 1$ as the steps closely follow the arguments for $n = 1$. Consider the second interval $[t_1, t_2]$, with
\begin{equation} \label{eq:uc2}
	u^{c}_{2}(t) = -\frac{1}{\Delta t_2} {g^{c}_{0}}^{\dagger}\left(\tilde{x}_1 + g_{0}^{uc} \alpha_2\right),
\end{equation}
where we recall that $\|\tilde{x}_1\|_{\infty} = \|x_1 - x_{tg}\|_{\infty} \leq \overline{e}_1$ in \eqref{eq:e1}. Substituting $\Delta t_2 = \frac{t_f}{4}$ and rearranging, the condition $\|u^{c}_{2}(t)\|_{\infty} \leq 1$ is satisfied if
$$
t_f \geq 4\|{g^{c}_{0}}^{\dagger}\|_{\infty} \left[\overline{e}_1 + \|g^{uc}_{0}\|_{\infty} \|\alpha_2\|_{\infty}\right],
$$
or, using \eqref{eq:e1} and rearranging,
\begin{equation} \label{eq:tf2}
e^{t_fD_S/2} \leq \frac{t_fD_S - c_2}{c_1} + 1,
\end{equation}
where we substitute $\Delta t_1 = t_f/2$ and the values of $c_1$ and $c_2$ from \eqref{eq:c1}, \eqref{eq:c2}. Using the definition of the function $h$ in \eqref{eq:h}, condition \eqref{eq:tf2} can be rewritten as $h(t_f) \leq 0$, which is true when
\begin{equation} \label{eq:tf_int}
t_f \in [\underline{t}, \overline{t}]
\end{equation}
under the conditions \eqref{eq:ccond} in Lemma~\ref{lem:h}. Thus, the input constraint $u^{c}_{2} \in \Uc$ is satisfied when $t_f \in [\underline{t}, \overline{t}]$, forming the second part of condition \eqref{eq:tf_cond}.

To prove the bound \eqref{eq:en} for $n = 2$, we follow identical steps to the case when $n = 1$. We consider $v(t_1, t) = -\int_{t_1}^{t} f(x(\tau))\mathrm{d}\tau - \int_{t_1}^{t} d_g(x(\tau), x_0) u_2(\tau)\mathrm{d}\tau$ for $t \in [t_1, t_2]$ from \eqref{eq:v_def}. Similar to \eqref{eq:v01}, we can obtain the following bound:
\begin{align} \label{eq:v12}
\|v(t_1, t)\|_{\infty} &\leq (t-t_1) \left(\|f(x_0)\|_{\infty} + D_S \|x_{tg} - x_0\|_{\infty} \right) \nonumber \\
&+ D_S \int_{t_1}^{t} \|x_{tg} - x(\tau)\|_{\infty} \mathrm{d}\tau,
\end{align}
and use this bound to obtain the following expression similar to \eqref{eq:err_t3}:
\begin{align}
\|x_{tg} - x(t)\|_{\infty} &\leq c(t-t_1) + \|\tilde{x}_0\|_{\infty} + \frac{1}{4}\|g_{0}^{uc}\|_{\infty} \nonumber \\
&+ D_S \int_{t_1}^{t} \|x_{tg} - x(\tau)\|_{\infty} \mathrm{d}\tau
\end{align}
when $t \in [t_1, t_2]$ and we use $\|\alpha_2\|_{\infty} = \frac{1}{4}$. Following similar steps to the case for $n = 1$ which involve Gr{\"o}nwall's inequality, we obtain the following bound for the error at $t_2$:
\begin{equation} \label{eq:e2}
\|x_{tg} - x_2\|_{\infty} \leq \overline{e}_2 \coloneqq \frac{c}{D_S} \left(e^{\Delta t_2 D_S} - 1\right).
\end{equation}

For larger $n$, the proof follows in a similar way, except condition \eqref{eq:tf_int} actually implies $\|u_{n}^{c}(t)\|_{\infty} \leq 1$ for all $n$. To show this, we first consider $u^{c}_{3}$ and note that $\|\alpha_3\|_{\infty} = \frac{1}{2}\|\alpha_2\|_{\infty}$, $\Delta t_3 = \frac{1}{2}\Delta t_2$ and $\overline{e}_2 \leq \frac{1}{2}\overline{e}_1$ from Lemma~\ref{lem:err}. Then,
\begin{align*}
\|u^{c}_{3}(t)\|_{\infty} &\leq \frac{1}{\Delta t_3} \|{g^{c}_{0}}^{\dagger}\|_{\infty}\left[\overline{e}_2 + \|g^{uc}_{0}\|_{\infty} \|\alpha_{3}\|_{\infty}\right] \\
&\leq \frac{2}{\Delta t_2} \|{g^{c}_{0}}^{\dagger}\|_{\infty} \left[\frac{1}{2}\overline{e}_1 + \|g^{uc}_{0}\|_{\infty} \frac{1}{2} \|\alpha_2\|_{\infty}\right] \\
&= \frac{1}{\Delta t_2} \|{g^{c}_{0}}^{\dagger}\|_{\infty} \left[\overline{e}_1 + \|g^{uc}_{0}\|_{\infty} \|\alpha_2\|_{\infty}\right] \\
&= \|u_{2}^{c}(t)\|_{\infty} \leq 1
\end{align*}
under \eqref{eq:tf_int}. We then have a bound similar to \eqref{eq:e2} at $t_3$, and identical arguments can be used to prove the bound
\begin{equation} \label{eq:xtbound}
\|x_{tg} - x(t)\|_{\infty} \leq \frac{c}{D_S} \left(e^{(t - t_{n-1})D_S} - 1 \right), ~t \in [t_{n-1}, t_n],
\end{equation}
and in particular,
$$
\|x_{tg} - x_n\|_{\infty} \leq \overline{e}_n \coloneqq \frac{c}{D_S} \left(e^{\Delta t_n D_S} - 1\right)
$$
for all $n$, as required in \eqref{eq:en}.

Summarizing, conditions \eqref{eq:tf1} and \eqref{eq:tf_int} can be combined to form \eqref{eq:tf_cond}. Under \eqref{eq:tf_cond} and \eqref{eq:ccond}, the input constraint $u_{c}^{n} \in \Uc$ is then satisfied for all $n$. Then, the bound \eqref{eq:en} holds for all $n$, thus concluding the proof.
\end{proof}

Theorem~\ref{thm:bound} thus develops conditions on $t_f$ under which the input constraint \eqref{eq:SetUc} is satisfied by the sequence of inputs $u^{c}_{n}$ in \eqref{eq:ucn}. Using this bound, we prove that Problem~\ref{prob} is solved by the method presented in Section~\ref{sec:Method}, when the sequence of inputs is applied until some $\overline{n}$.

\begin{figure}[!t]
	\centering
	\includegraphics[width = 0.35\textwidth]{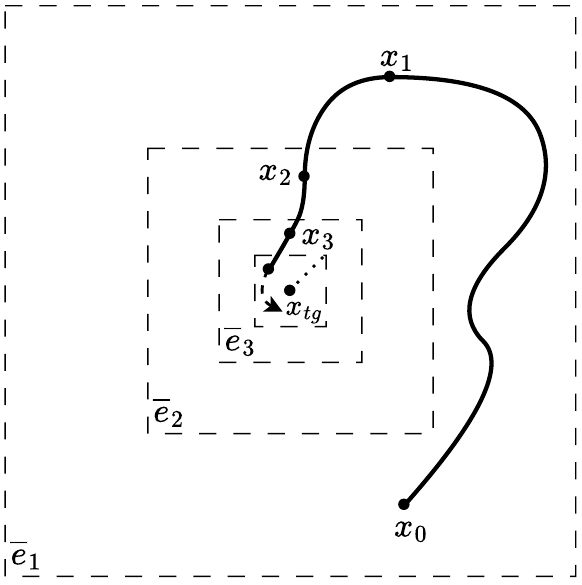}
	\caption{Illustrating the convergence behavior based on Theorem~\ref{thm:bound} and Corollary~\ref{cor:conv}. At every $t_n$, the bound on the error between $x(t)$ and $x_{tg}$ is halved.}
	\label{fig:conv}
\end{figure}

\begin{corollary}[Convergence] \label{cor:conv}

{Assume conditions \eqref{eq:ccond} and \eqref{eq:tf_cond} hold and let $\varepsilon > 0$. Apply the sequence of inputs \eqref{eq:ucn} until some $n = \overline{n}-1$, where $\overline{n}$ satisfies $t_f - t_{\overline{n}-1} \leq \Delta t_{n_1}$ and $n_1$ is the smallest number such that $\overline{e}_{n_1} \leq \epsilon$. In the remaining interval $[t_{\overline{n}-1}, t_f]$, apply the last input $u^{c}_{\overline{n}}$. Then, this sequence of inputs ensures $\|x_{tg} - x(t_f)\|_{\infty} \leq \varepsilon$.}

\end{corollary}

\begin{figure*}[!t]

\centering
\begin{subfigure}{0.49\textwidth}
	\centering
	\includegraphics[width = 0.8\textwidth]{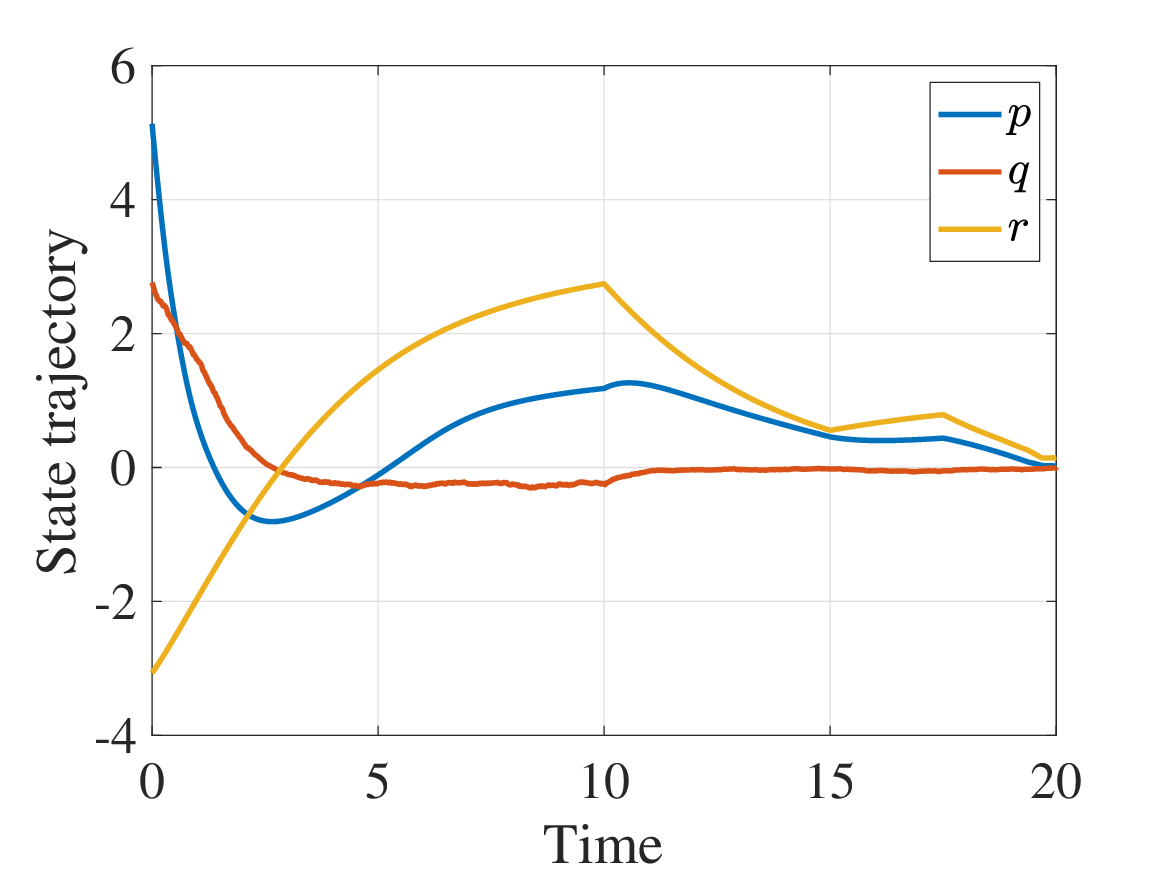}
	\caption{State trajectories}
	\label{fig:state_ADMIRE}
\end{subfigure}
\hfill
\begin{subfigure}{0.49\textwidth}
	\centering
	\includegraphics[width = 0.8\textwidth]{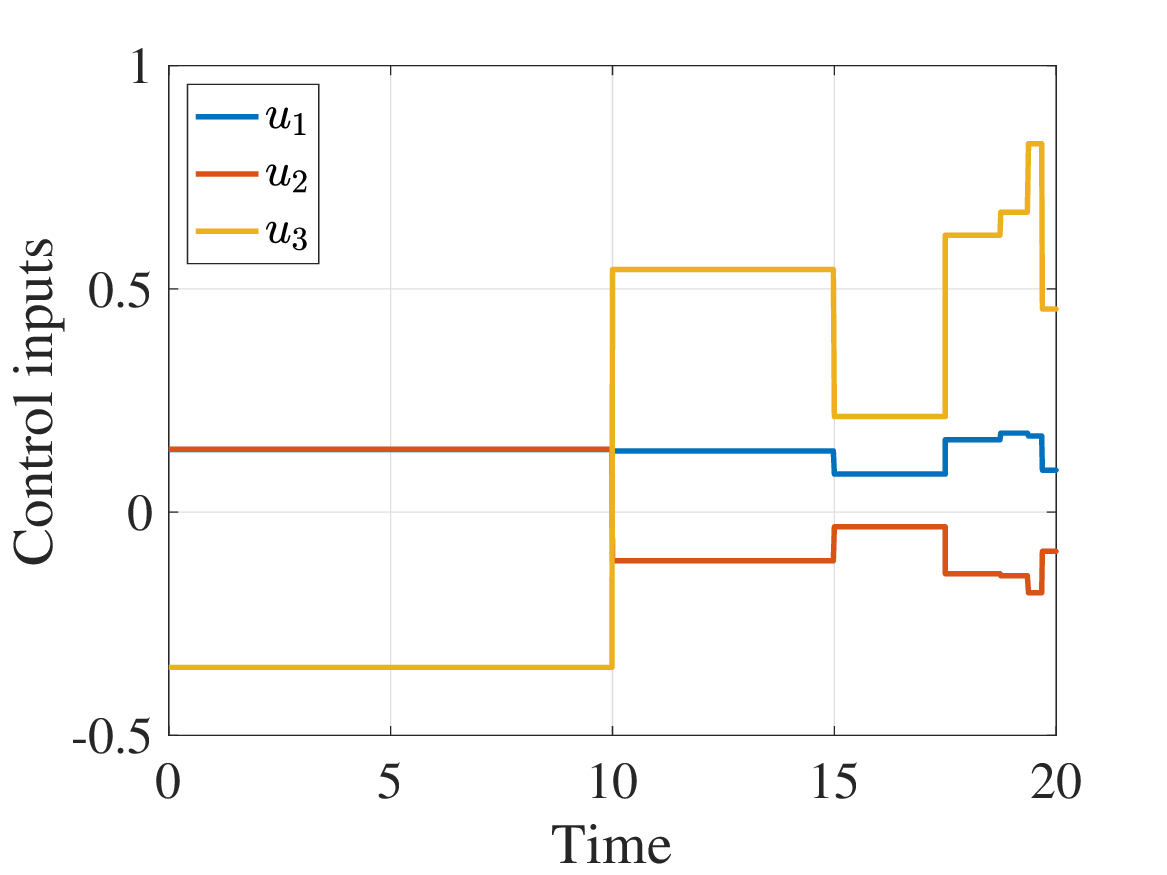}
	\caption{Control inputs}
	\label{fig:input_ADMIRE}
\end{subfigure}
\caption{State trajectories and control inputs for the ADMIRE fighter-jet model, reaching the target state $x_{tg}$ despite uncontrolled effects from the canard wing $u_1$.}
\label{fig:ADMIRE}

\end{figure*}

\begin{proof}

{Since the input $u^{c}_{\overline{n}}$ is applied for the remaining interval $[t_{\overline{n}-1}, t_f]$, bound \eqref{eq:xtbound} holds for the entire interval $[t_{\overline{n}-1}, t_f]$ rather than just the shorter interval $[t_{\overline{n}-1}, t_{\overline{n}}]$.} Then, \eqref{eq:xtbound} can be written as
\begin{align*}
\|x_{tg} - x(t)\|_{\infty} &\leq \frac{c}{D_S}\left(e^{(t - t_{\overline{n}-1})D_S} - 1 \right), ~t \in [t_{\overline{n}-1}, t_f] \\
\text{or } \|x_{tg} - x(t_f)\|_{\infty} &\leq \frac{c}{D_S}\left(e^{(t_f - t_{\overline{n}-1})D_S} - 1 \right) \text{ at } t = t_f.
\end{align*}
Next, we use the fact that $t_f - t_{\overline{n}-1} \leq \Delta t_{n_1}$, where $n_1$ is the smallest number which satisfies $\overline{e}_{n_1} \leq \varepsilon$ and exists from Lemma~\ref{lem:err}. Substituting $t_f - t_{\overline{n}-1} \leq \Delta t_{n_1}$ above, we have
\begin{equation}
\|x_{tg} - x(t_f)\|_{\infty} \leq \frac{c}{D_S} \left(e^{\Delta t_{n_1} D_S} - 1\right) = \overline{e}_{n_1} \leq \varepsilon
\end{equation}
as required, using the definition of $\overline{e}_n$ in \eqref{eq:errn}.
\end{proof}

{We remark that while \eqref{eq:ccond} and \eqref{eq:tf_cond} guarantee the constraint $u^c \in \Uc$ is satisfied, these conditions are sufficient and not necessary. Indeed, in practice, these conditions may be very constraining and may not be satisfied by practical, real-world control systems. However, the analysis in Theorem~\ref{thm:bound} is quite conservative and even choices of $t_f$ that do not satisfy \eqref{eq:tf_cond} may lead to a viable sequence of inputs \eqref{eq:ucn} that satisfy Problem~\ref{prob}.
}

An illustration of the convergence behavior from Theorem~\ref{thm:bound} and Corollary~\ref{cor:conv} is provided in Fig.~\ref{fig:conv}. At every $t_n$, the error between $x(t)$ and $x_{tg}$ is halved, which is a consequence of the design of $\{t_n\}$ in Fig.~\ref{fig:Horizon} and $\alpha_n$ in \eqref{eq:alpha_n}. In the next section, we present an example illustrating the use of this method on the model of a fighter jet.

\section{Illustrative Example} \label{sec:Example}
We consider the dynamics of the ADMIRE fighter-jet model subjected to nonlinear wind effects, as considered in \cite{PO26a}. These dynamics are obtained from a linearized model established in \cite{SDRA}, and models of nonlinear wind effects based on \cite{WWS24}. The dynamics can be written as
\begin{equation} \label{eq:ADMIRE}
\dot{x}(t) = Ax(t) + f_w(x(t)) + B^{c}u^{c}(t) + B^{uc}u^{uc}(t)
\end{equation}
where
\begin{align*}
x &= \begin{bmatrix} p \\ q \\ r \end{bmatrix}; \hspace{0.5em} A = \begin{bmatrix} -0.9967 & 0 & 0.6176 \\ 0 & -0.5057 & 0 \\ -0.0939 & 0 & -0.2127 \end{bmatrix}; \\
f_w(x) &= \frac{1}{2}\left[\sin(p)\cos^2(p), -\sin(2q), 1\right]^{\top}; \\
B^c &= \begin{bmatrix} -4.2423 & 4.2423 & 1.4871 \\ -1.2735 & -1.2735 & 0.0024 \\ -0.2805 & 0.2805 & -0.8823 \end{bmatrix}; \hspace{0.5em} B^{uc} = \begin{bmatrix} 0 \\ 1.6532 \\ 0 \end{bmatrix}.
\end{align*}
The states in this model consist of roll, pitch and yaw rates. The three controlled inputs in $u^c$ correspond to the left and right elevons and rudder, and the uncontrolled input $u^{uc}$ corresponds to the canard, in line with prior work \cite{BO23, PO26a}. The wind effects are chosen as sinusoidal and constant expressions based on \cite{WWS24}, such that their Lipschitz constant is no more than $1$. In particular, these dynamics satisfy conditions \eqref{eq:f_cond} and \eqref{eq:g_cond} everywhere in $\X = \mathbb{R}^3$, with $D_f = \|A\|_{\infty} + 1 = 2.6143$ and $D_g = 0$, since \eqref{eq:ADMIRE} is linear in the input.

The linear driftless approximation \eqref{eq:DftApprox} is then formed by the matrices $g_{0}^{c} = B^c$ and $g_{0}^{uc} = B^{uc}$. We select a final time $t_f = 20$, a randomly chosen initial state $x_0 = [5.13 ~~ 2.76 ~~ -\!3.07]^{\top}$ and the target state $x_{tg} = [0~0~0]^{\top}$. The uncontrolled canard input is allowed to take on any values in the range $[-1, 1]$ according to constraint \eqref{eq:SetUuc}. We terminate the procedure in Section~\ref{sec:Method} at just $n = 8$, i.e., partitioning the time horizon and designing the sequence \eqref{eq:ucn} only until $n = 8$.

The results of our procedure are shown in Fig.~\ref{fig:ADMIRE}. We can see that the state trajectories settle at the target state at exactly the desired time $t_f = 20$, despite uncontrolled effects of the canard input $u_1$. This behavior is achieved while the controlled inputs stay within prescribed constraints \eqref{eq:SetUc}. The piecewise nature of state and input trajectories is also clear, based on the partition of the time horizon and the sequence \eqref{eq:ucn}. 

A particular advantage of this method is its low computational effort, taking under $0.5$ seconds on a modern laptop, and even under $0.2$ seconds for constant uncontrolled inputs as is common in some fault-tolerant control approaches \cite{AH19}. In contrast, some modern techniques for nonlinear control require significantly more computational effort through large-scale optimization problems, as discussed in \cite{PO25b}.

\section{Conclusions}
In this paper, we presented a method to drive a nonlinear system to a target state, under input constraints and when the system suffers a partial loss in control authority. We designed a sequence of controlled inputs that are based on a linear driftless approximation to the original nonlinear dynamics. We then partitioned the finite time horizon into successively shorter intervals, and proved that the sequence of designed controlled inputs results in bounded error from the target state in the original nonlinear system. As the length of each partition converges to zero, we proved that the error from the target state also reduces to zero. Simultaneously, we developed conditions on the length of the horizon guaranteeing input constraints are satisfied. Using the model of a fighter jet losing authority over one of its inputs, we demonstrated how the designed inputs achieve the target state, despite the effect of the uncontrolled input. Future work will be dedicated to investigating more general sequences $\alpha_n$ in \eqref{eq:alpha_n}, and understanding how they impact the convergence result in Theorem~\ref{thm:bound}. Further, we intend to develop a learning-based framework for such a method, where the linear driftless approximation can be learned from data over a small interval, before designing a sequence of controlled inputs.

\balance
\bibliographystyle{IEEEtran}
\bibliography{references}

\end{document}